\documentclass[10pt]{amsart}

\usepackage[backend=bibtex, maxnames=7]{biblatex}
\bibliography{non_contracting}

\usepackage{graphicx}

\usepackage{color}

\usepackage{caption}
\usepackage{enumitem}

\usepackage{amsmath}
\usepackage{amsthm}
\usepackage{amssymb}
\usepackage{amscd}

\usepackage{pgf}
\usepackage{tikz}
\usetikzlibrary{shapes,arrows,automata}

\newtheorem{thm}{Theorem} 
\newtheorem{prop}[thm]{Proposition}
\newtheorem{lem}[thm]{Lemma}

\theoremstyle{definition}

\newtheorem{dfn}[thm]{Definition}

\newcommand{\R}{\mathbb{R}}

\newcommand{\N}{\mathbb{N}}

\renewcommand{\le}{\leqslant}
\renewcommand{\ge}{\geqslant}

\renewcommand{\>}{\rangle}

\newcommand{\rest}[1]{\!\mid_{#1}}

\newcommand{\zz}{\tiny $0|0$}
\newcommand{\zo}{\tiny $0|1$}
\newcommand{\oz}{\tiny $1|0$}
\newcommand{\oo}{\tiny $1|1$}

\DeclareMathOperator{\Aut}{Aut}

\begin{document}

\title{Some non-contracting automata groups}
\author{Nick Davis}
 
\address{Department of Mathematics and Statistics\\
University of Melbourne\\
Parkville VIC 3010 \\
Australia
}
\email{nickd@unimelb.edu.au} 

\author{Murray Elder}

\address{School of Mathematical and Physical Sciences\\
The University of Newcastle\\
Callaghan NSW 2308\\
Australia
}
\email{murrayelder@gmail.com}

\author{Lawrence Reeves}
\address{Department of Mathematics and Statistics\\
University of Melbourne\\
Parkville VIC 3010 \\
Australia
}
\email{lreeves@unimelb.edu.au} 

\thanks{The first author was supported by an Australian Postgraduate Award. The second and third authors
 are supported by
Australian Research Council grants  FT110100178 and DP1096912.
}

\subjclass[2010]{20F65}
\date{13 November 2013}
\keywords{automata groups, self-similar groups, contracting actions}

\begin{abstract} We add to the classification of groups generated by 3-state automata over a 2 letter alphabet given 
by Bondarenko \textit{et al.}~\cite{BGKMNSS}, by showing that a number of the groups in the classification are non-contracting. 
We show that the criterion we use to prove a self-similar action is non-contracting also implies that the associated 
self-similarity graph introduced by Nekrashevych is non-hyperbolic.
\end{abstract}

\maketitle

\section{Introduction}

In \cite{BGKMNSS} a list of automata groups
 generated by 3-state
automata over a 2-letter alphabet is given and a great deal of information is listed for each.
Amongst the data given for each group was whether the group was contracting or non-contracting.
For  ten automata  the classification did not determine whether or not the group was
 contracting.
In the numbering system of  \cite[page 14]{BGKMNSS}  the ten automata are:
$$
749, 861, 882, 887, 920, 969, 2361, 2365, 2402, 2427.
$$

Later Muntyan \cite{M} showed that three of these are isomorphic to other groups in the classification, 
specifically $920\cong   2401$, $2361\cong 939$, and  $2365\cong 939$, where the groups $939$ and 
$2401$ are listed as non-contracting in  \cite{BGKMNSS}.

The purpose of this note is to show that all of the automata groups listed above are  non-contracting.
We first establish a criterion for a group to be non-contracting, and then apply it in each case.

We refer to \cite{N} for the basic definitions of self-similar actions and automata groups. 
Automata groups are examples of self-similar actions.
The automata to be considered all have three states, labelled $a$ $b$ and $c$, and a two letter 
alphabet $X=\{0,1\}$.
The automata are represented by a Moore diagram, which is given below for each automaton.
Each state defines an element in $\Aut(X^\omega)$ and the group defined by the automaton is
$G=\<a,b,c\>\subset\Aut(X^\omega)$. For $g\in G$ and a finite word $v\in X^*$, the restriction of 
$g$ to $v$ is denoted $g\rest{v}$. The basic properties of the action that we will make use of are:
$$
(gh)\rest{v}=g\rest{h(v)}h\rest{v} \quad g({uv})=g(u)g\rest{u}(v) \quad g\rest{uv}=(g\rest{u})\rest{v}
$$

Recall the following definition from \cite{N}, 
\begin{dfn}
A self-similar action $G\le \Aut(X^\omega)$ is called \emph{contracting}
if there exists a finite subset $\mathcal{N}\subseteq G$ such that for all $g\in G$ there exists $k\in\N$ such that
$g\rest{v}\in\mathcal{N}$ for all $v\in X^{*}$ with $\ell(v)\ge k$.
The minimum such $\mathcal{N}$ is called the \emph{nucleus} of the action.
\end{dfn}


We make use of the following criterion, which was used 
 in \cite{BGKMNSS}. 
For example,  it is used to show that 744 is non-contracting, and many times after that.

\begin{lem}\label{lem:criterion}
Let $G\le \Aut(X^\omega)$ be a self-similar action.
Suppose that there exist $g\in G$ and   $v\in X^*$ such that 
\begin{enumerate}
\item $g\rest{v}=g$
\item $g(v)=v$
\item $g$ has infinite order
\end{enumerate}
Then $G$ is non-contracting.
\end{lem}
\begin{proof}
Assume for induction that $g\rest{v^k}=g$ and $g(v^k)=v^k$ for $k\geq 1$. Then
 $g\rest{v^{k+1}}=g\rest{v^kv}=(g\rest{v^k})\rest{v}=g\rest{v}=g$ and $g(v^{k+1})=g(v)g\rest{v}(v^k)=vg(v^k)=vv^k$.
 
 Next assume for induction that
$g^n\rest{v^k}=g^n$ for $n\geq 1$ and fixed $k$.
Then $g^{n+1}\rest{v^k}=g\rest{g^n(v^k)}g^n\rest{v^k}=
g\rest{v^k}g^n=gg^n$.

It follows that a nucleus must contain $g^n$ for infinitely many $n$ and so, since $g$ has infinite order, the action is not contracting.

Alternatively, though less directly, the lemma follows from Proposition \ref{lem:nonhyp} below and Theorem 3.8.6 of  \cite{N}.
\end{proof}

In the next section we apply this criterion to the ten automata listed above. In Section \ref{sec:hyperbolic} we prove that a self-similar group satisfying this criterion has a non-hyperbolic self-similarity graph. 

\section{The automata}\label{sec:automata}

For each of the ten automata not listed as contracting or non-contracting in \cite{BGKMNSS} we give an element $g\in G$ and
a word $v\in\{0,1\}^*$ with  the (easily verifiable) property that $g(v)=v$ and $g\rest{v}=g$. 
The Moore diagram of the automaton is given for reference. Active states are shaded in the diagram.
Then an argument is given to prove that
$g$ has infinite order. The criterion of the lemma above then applies.
The approach to showing that $g$ has infinite order is to
 find another string $v'$ that is not fixed by any power of $g$. 
 We found the candidates for suitable elements and strings using some simple computer code and observing various patterns.

 It is convenient to 
introduce the equivalence relation on $\{0,1\}^\omega$ given by left shift equivalence, that is,
$u\sim v$ if  there are finite prefixes $u'$ and $v'$ of $u$ and $v$ respectively, and
 $w\in \{0,1\}^\omega$ such that $u=u'w$ and $v=v'w$.

For a finite word $u\in\{0,1\}^*$, we denote by $u^\infty$ the element of $\{0,1\}^\omega$ formed
by repeating $u$ infinitely many times.

\subsection{Automaton 749}

\mbox{}

\begin{minipage}{0.5\textwidth}
\begin{gather*}a^2bc(0100)=0100 \\
 (a^2bc)\rest{0100}=a^2bc
\end{gather*}
\end{minipage}
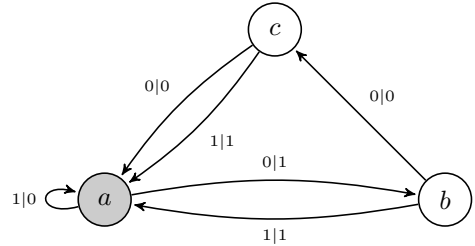
\parbox{0.4\textwidth}{
\tikzset{every state/.style={minimum size=2em}} 
\begin{tikzpicture}[->,>=stealth',shorten >=1pt,auto,node distance=3.2cm,
                    semithick]

  \node[state, fill=gray!40] (A)                    {$a$};
  \node[state]         (C) [above right of=A] {$c$};
  \node[state]         (B) [below right of=C] {$b$};

  \path (A) edge  [loop left] node {\oz} (A)           
                 edge  [bend left=10]            node {\zo} (B)
        (B) edge  [bend left=10] node    {\oo} (A)
            edge              node [swap] {\zz} (C)
         (C) edge [bend right=10] node [swap] {\zz} (A)
           (C) edge [bend left=10] node {\oo} (A);
        \end{tikzpicture}
}
\bigskip

To see that $g=a^2bc$ has infinite order we consider the string $0^\infty$.
Observe that since 
$g\rest{000}=babc$, $babc(000)=101$, {and}
$(babc)\rest{000}=babc$,
we have
$g(0^\infty)=001(101)^\infty$.
Then note that
$
a\rest{101}=b\rest{101}=c\rest{101}=a$  {and}  $a^4(101)=101$.
It follows that for any $n\ge 1$, $g^n(0^\infty)=u_n(101)^\infty$
where $u_1=001$ and $u_{n}=g(u_{n-1}101)$.
In other words $g^n(0^\infty)$ is left-shift equivalent to $(101)^\infty$.
We now note that $g^{-1}(0^\infty)$ is not of this form, which 
establishes that $g$ has infinite order.
Observe that
\begin{gather*}
g^{-1}\rest{0000}=a^{-1}b^{-1}a^{-2}  \quad g^{-1}(0000)=0011\\
a^{-1}b^{-1}a^{-2}\rest{0000}=a^{-1}b^{-1}a^{-2}   \quad     a^{-1}b^{-1}a^{-2}(0000)=1011
\end{gather*}
Therefore
$g^{-1}(0^\infty)=0011(1011)^\infty$.

\subsection{Automaton 861}

\mbox{}

\begin{minipage}{0.5\textwidth}
\begin{gather*}
c(010)=010\\ c\rest{010}=c
\end{gather*}

\end{minipage}
\parbox{0.4\textwidth}{
\tikzset{every state/.style={minimum size=2em}} 
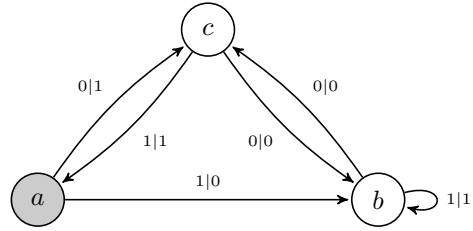
\begin{tikzpicture}[->,>=stealth',shorten >=1pt,auto,node distance=3.2cm,
                    semithick]

  \node[state, fill=gray!40] (A)                    {$a$};
  \node[state]         (C) [above right of=A] {$c$};
  \node[state]         (B) [below right of=C] {$b$};

  \path (A) edge  [bend left=10] node {\zo} (C)           
                 edge             node {\oz} (B)
        (B) edge  [loop right] node    {\oo} (B)
            edge     [bend right=10]         node [swap] {\zz} (C)
         (C) edge [bend right=10] node [swap] {\zz} (B)
           (C) edge [bend left=10] node {\oo} (A);
        \end{tikzpicture}
}
\bigskip


Since $x\rest{11}=b$ for any $x\in\{a,b,c\}$ and
$b(1^\infty)=1^\infty$, it follows that $c^n(1^\infty)\sim 1^\infty$ for any $n\ge 0$.
But $c^{-1}(1^\infty)=(10)^\infty$, so $c$ has  infinite order.

\subsection{Automaton 882}
\mbox{}

\begin{minipage}{0.5\textwidth}
\begin{gather*}
acacbc(11)=11 \\ (acacbc)\rest{11}=acacbc
\end{gather*}
\end{minipage}
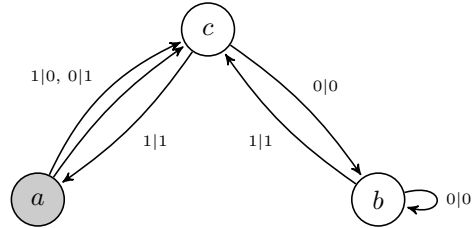
\parbox{0.4\textwidth}{
\tikzset{every state/.style={minimum size=2em}} 
\begin{tikzpicture}[->,>=stealth',shorten >=1pt,auto,node distance=3.2cm,
                    semithick]

  \node[state, fill=gray!40] (A)                    {$a$};
  \node[state]         (C) [above right of=A] {$c$};
  \node[state]         (B) [below right of=C] {$b$};

  \path (A) edge  [bend left=20] node {\oz, \zo} (C)           
                 edge  [bend left=10]            node {} (C)
        (B) edge  [loop right] node    {\zz} (B)
            edge     [bend left=10]         node  {\oo} (C)
         (C) edge [bend left=10] node  {\zz} (B)
           (C) edge [bend left=10] node {\oo} (A);
        \end{tikzpicture}
}
\bigskip

To show that $g=acacbc$ has infinite order we use the following lemma to
conclude that   $g^{2^n}(0^\infty)=0^{2n+1}110^\infty$ for all $n\le 1$.

\begin{lem}\mbox{}
\begin{enumerate}
\item $g^{2^n}(0^{2n+1})=0^{2n+1}$
\item $g^{2^n}\rest{0^{2n+1}}=cacb$
\item $cacb(0^\infty)=110^{\infty}$
\end{enumerate}
\end{lem}
\begin{proof}
For the third part
\begin{align*}
cacb(0^\infty)&=cacb(00)(cacb)\rest{00}(0^\infty)=11cbbb(0^\infty)\\
&=11cb(0^\infty)=11cb(0)(cb)\rest{0}(0^\infty)=110b^2(0^\infty)=110^\infty
\end{align*}

We prove the first and second by induction on $n$. 
Note first that $b^2$ is the identity in the group, as can be seen from the automaton for $b^2$.
\smallskip

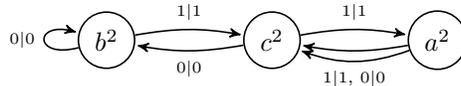
\begin{center}
\tikzset{every state/.style={minimum size=2em}} 
\begin{tikzpicture}[->,>=stealth',shorten >=1pt,auto,node distance=2.2cm,
                    semithick]

  \node[state] (A)                    {$a^2$};
  \node[state]         (C) [left of=A] {$c^2$};
  \node[state]         (B) [left  of=C] {$b^2$};

  \path (A) edge  [bend left=20] node {\oo, \zz} (C)           
                 edge  [bend left=10]            node {} (C)
        (B) edge  [loop left] node    {\zz} (B)
            edge     [bend left=10]         node  {\oo} (C)
         (C) edge [bend left=10] node  {\zz} (B)
           (C) edge [bend left=10] node {\oo} (A);
        \end{tikzpicture}
\end{center}

We have $g(0)=0$ and $g\rest{0}=cacbbb=cacb$.
Then inductively,
\begin{align*}
g^{2^{n+1}}(0^{2^n+3})&=g^{2^n}g^{2^n}(0^{2^{2n+1}}00)\\
&= g^{2^n}(0^{2n+1}cacb(00))= g^{2^n}(0^{2n+1}11)\\
&=0^{2n+1}cacb(11)=0^{2n+1}00\\
g^{2^{n+1}}\rest{0^{2^n+3}}&=(g^{2^n}g^{2^n})\rest{0^{2^n+3}}= g^{2^n}\rest{g^{2^n}(0^{2n+3})}g^{2^n}\rest{0^{2^n+3}}\\
&= g^{2^n}\rest{0^{2n+1}11}g^{2^n}\rest{0^{2^n+3}}
= (g^{2^n}\rest{0^{2n+1}})\rest{11}(g^{2^n}\rest{0^{2^n+1}})\rest{00}\\
&=(cacb)\rest{11}(cacb)\rest{00}=bbcacbbb=cacb
\end{align*}

\end{proof}

\subsection{Automaton 887}
\mbox{}

\begin{minipage}{0.5\textwidth}
\begin{gather*}
bc(00)=00\\(bc)\rest{00}=bc
\end{gather*}
\end{minipage}
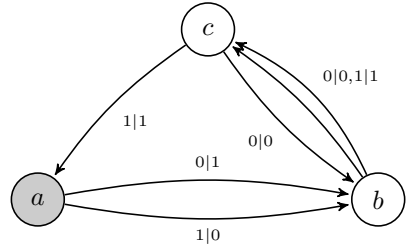
\parbox{0.4\textwidth}{
\tikzset{every state/.style={minimum size=2em}} 
\begin{tikzpicture}[->,>=stealth',shorten >=1pt,auto,node distance=3.2cm,
                    semithick]

  \node[state, fill=gray!40] (A)                    {$a$};
  \node[state]         (C) [above right of=A] {$c$};
  \node[state]         (B) [below right of=C] {$b$};

  \path (A) edge  [bend left=10] node {\zo} (B)           
                 edge  [bend right=10]            node [swap] {\oz} (B)
        (B) edge  [bend right=10] node    {} (C)
            edge     [bend right=20]         node [swap]  {\zz,\oo} (C)
         (C) edge [bend right=10] node [swap]  {\zz} (B)
           (C) edge [bend right=10] node {\oo} (A);
        \end{tikzpicture}
}
\bigskip

To establish that $bc$ has infinite order we prove  the following.

\begin{lem}
For all $n\ge 1$, $(bc)^n(1^\infty)\neq1^\infty$.
\end{lem}

\begin{proof}
Since $bc(1)=1$ and $(bc)\rest{1}=ca$, we have
$(bc)^n(1^\infty)=1(ca)^n(1^\infty)$ and it suffices to show that 
$(ca)^n(1^\infty)\neq 1^\infty$. 
We show  that for $n\ge 2$
\begin{enumerate}
\item $(ca)^{4n}(111)=111$ and $(ca)^{4n}(110)=110$
\item $(ca)^{2^n}\rest{111}=(ca)^{2^{n-1}}$ 
\item $ (ca)^{2^n}(1^\infty)=(111)^{n-1}(1010)1^\infty$ 
\end{enumerate}

It's clear that the third claim implies that $(ca)^n(1^\infty)\neq 1^\infty$ for all $n\ge 1$.

The first claim follows from $(ca)^4(111)=111$ and $(ca)^4(110)=110$.

For the second claim, note first that $a$, $b$ and $c$ all have order 2, as can be seen from
the automaton for $\<a^2,b^2,c^2\>$.
\medskip

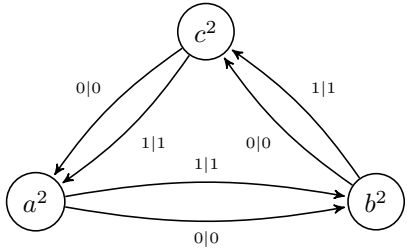
\tikzset{every state/.style={minimum size=2em}} 
\begin{tikzpicture}[->,>=stealth',shorten >=1pt,auto,node distance=3.2cm,
                    semithick]

  \node[state] (A)                    {$a^2$};
  \node[state]         (C) [above right of=A] {$c^2$};
  \node[state]         (B) [below right of=C] {$b^2$};

  \path (A) edge  [bend left=10] node {\oo} (B)           
                 edge  [bend right=10]            node [swap] {\zz} (B)
        (B) edge  [bend left=10] node    {\zz} (C)
            edge     [bend right=10]         node [swap]  {\oo} (C)
         (C) edge [bend right=10] node [swap]  {\zz} (A)
           (C) edge [bend left=10] node {\oo} (A);
        \end{tikzpicture}

Then $(ca)\rest{111}=aa=1$ and $(ca)\rest{110}=bb=1$.
Also $$(ca)^2\rest{111}=(ca)\rest{ca(111)}(ca)\rest{111}=(ca)\rest{011}(ca)\rest{111}=ca$$ 
$$
(ca)^4\rest{111}=(ca)^2\rest{(ca)^2(111)}(ca)^2\rest{111}=(ca)^2\rest{101}ca=caaaca=caca
$$
Inductively, for $n\ge 3$, 
\begin{align*}
(ca)^{2^{n}}\rest{111}&=((ca)^{2^{n-1}}(ca)^{2^{n-1}})\rest{111}
=(ca)^{2^{n-1}}\rest{(ca)^{2^{n-1}}(111)}(ca)^{2^{n-1}}\rest{111} \\
&=(ca)^{2^{n-1}}\rest{111}(ca)^{2^{n-2}} %
=(ca)^{2^{n-2}}(ca)^{2^{n-2}}=(ca)^{2^{n-1}}\\
\end{align*}
For the third claim, note that $ca(1^\infty)=0(bb)(1^\infty)=01^\infty$ and 
\begin{align*}
(ca)^4(1^\infty)&=(ca)^4(111)(ca)^4\rest{111}(1^\infty)=111(ca)^2(1^\infty)\\
&=111101(ca)^2\rest{111}(1^\infty)=111101(ca)(1^\infty)=11110101^\infty
\end{align*}
Then for $n\ge 3$
\begin{align*}
(ca)^{2^{n}}(1^\infty)&=(ca)^{2^{n}}(111)(ca)^{2^{n}}\rest{111}(1^\infty)
=111(ca)^{2^{n-1}}(1^\infty)\\  
&=111(111)^{n-1}(1010)1^\infty=(111)^{n}(1010)1^\infty
\end{align*}

\end{proof}

\subsection{Automaton 920}
\mbox{}

\begin{minipage}{0.5\textwidth}
\begin{gather*}
b(1)=1 \\
 b\rest{1}=b
\end{gather*}
\end{minipage}
\parbox{0.4\textwidth}{
\tikzset{every state/.style={minimum size=2em}} 
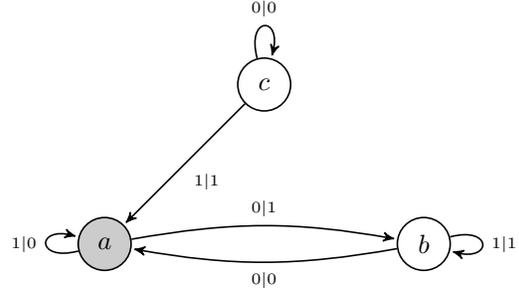
\begin{tikzpicture}[->,>=stealth',shorten >=1pt,auto,node distance=3cm,
                    semithick]

  \node[state, fill=gray!40] (A)                    {$a$};
  \node[state]         (C) [above right of=A] {$c$};
  \node[state]         (B) [below right of=C] {$b$};

  \path (A) edge [bend left=10] node {\zo} (B)           
                 edge    [loop left]            node {\oz} (A)
        (B) edge   [bend left=10] node    {\zz} (A)
            edge       [loop right]           node {\oo} (A)
           (C) edge node {\oo} (A)
                 edge [loop above] node {\zz} (C);
        \end{tikzpicture}
}
\bigskip

To show that $b$ has infinite order, consider  the inverse automaton:

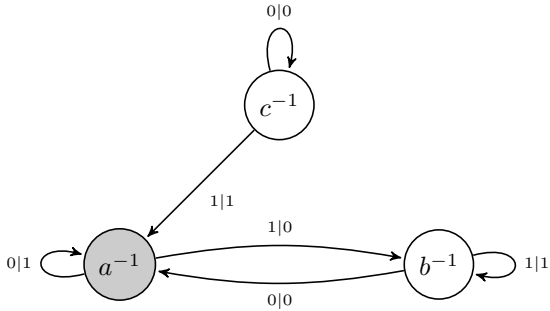
\tikzset{every state/.style={minimum size=1em}} 
\begin{tikzpicture}[->,>=stealth',shorten >=1pt,auto,node distance=3cm,
                    semithick]

  \node[state, fill=gray!40] (A)                    {$a^{-1}$};
  \node[state]         (C) [above right of=A] {$c^{-1}$};
  \node[state]         (B) [below right of=C] {$b^{-1}$};

  \path (A) edge [bend left=10] node {\oz} (B)           
                 edge    [loop left]            node {\zo} (A)
        (B) edge   [bend left=10] node    {\zz} (A)
            edge       [loop right]           node {\oo} (A)
           (C) edge node {\oo} (A)
                 edge [loop above] node {\zz} (C);
        \end{tikzpicture}

Since $a^{-1}\rest{1}=b^{-1}\rest{1}=b^{-1}$ and $b^{-1}(1)=1$, it follows that 
 $b^{-n}(01^\infty)\sim1^\infty$. But 
 $b(01^\infty)=0^\infty$, so $b$ has infinite order.

Note that according to \cite{N} this group is isomorphic to that of automaton $2401$
which is non-contracting.

\subsection{Automaton 969}
\mbox{}

\begin{minipage}{0.5\textwidth}
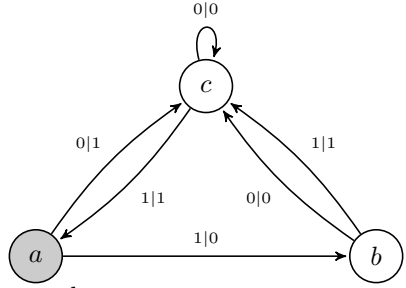
\begin{gather*}
c(0)=0\\c\rest{0}=c
\end{gather*}
\end{minipage}
\parbox{0.4\textwidth}{
\tikzset{every state/.style={minimum size=2em}} 
\begin{tikzpicture}[->,>=stealth',shorten >=1pt,auto,node distance=3.2cm,
                    semithick]

  \node[state, fill=gray!40] (A)                    {$a$};
  \node[state]         (C) [above right of=A] {$c$};
  \node[state]         (B) [below right of=C] {$b$};

  \path (A) edge  [bend left=10] node {\zo} (C)           
                 edge              node {\oz} (B)
        (B) edge   [bend left=10] node    {\zz} (C)
            edge      [bend right=10]        node [swap] {\oo} (C)
           (C) edge [bend left=10] node {\oo} (A)
                 edge [loop above] node {\zz} (C);
        \end{tikzpicture}
}

That $c$ has infinite order follows from the next lemma.

\begin{lem}
For $n\ge1$, $c^n((101)^\infty)\sim\begin{cases} (100)^\infty & n \text{ even}\\ (011)^\infty & n \text{ odd}\end{cases}$
\end{lem}
\begin{proof}
Note that $c((101)^\infty)=11c((110)^\infty)=11(100)^\infty$.
If $u\sim(100)^\infty$, then $c(u)\sim(011)^\infty$. If $u\sim(011)^\infty$, then $c(u)\sim(100)^\infty$.
Both statements  follow from the observation that for any generator $x\in\{a,b,c\}$, $x\rest{10}=c$.
\end{proof}

Finally, observe that $c^{-1}((101)^\infty)=1^\infty$.  
This together with the lemma imply that $c$ has infinite order.

\subsection{Automaton 2361}
\mbox{}

\begin{minipage}{0.5\textwidth}
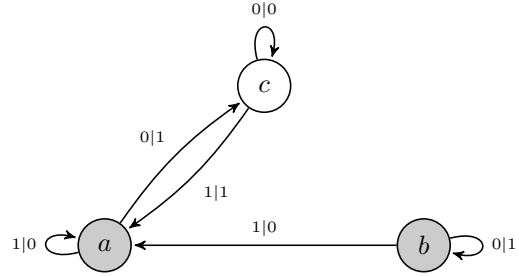
\begin{gather*}
c(0)=0\\c\rest{0}=c
\end{gather*}
\end{minipage}
\parbox{0.4\textwidth}{
\tikzset{every state/.style={minimum size=2em}} 
\begin{tikzpicture}[->,>=stealth',shorten >=1pt,auto,node distance=3cm,
                    semithick]

  \node[state, fill=gray!40] (A)                    {$a$};
  \node[state]         (C) [above right of=A] {$c$};
  \node[state, fill=gray!40]         (B) [below right of=C] {$b$};

  \path (A) edge  [bend left=10] node {\zo} (C)           
                 edge  [loop left]            node {\oz} (A)
        (B) edge   [loop right] node    {\zo} (B)
            edge                      node [swap] {\oz} (A)
           (C) edge [bend left=10] node {\oo} (A)
                 edge [loop above] node {\zz} (C);
        \end{tikzpicture}
}

\mbox{}

Observe that
$
a(0^\infty)=10^\infty$ and $  c(0^\infty)=0^\infty
$.
Therefore, for all $n\ge 0$, $c^n(10^\infty)\sim0^\infty$.
Also, $c^{-1}(10^\infty)=1^\infty$. It follows that $c$ has infinite order.

Note that according to \cite{M} this group is isomorphic to that of automaton $939$
which is non-contracting.

\subsection{Automaton 2365}
\mbox{}

\begin{minipage}{0.5\textwidth}
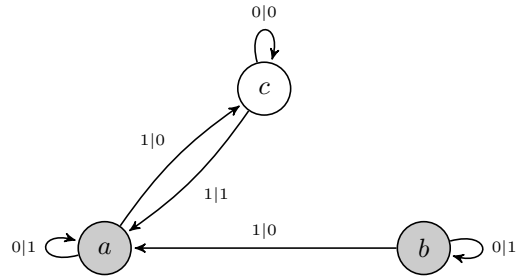
\begin{gather*}
c(0)=0\\c\rest{0}=c
\end{gather*}
\end{minipage}
\parbox{0.4\textwidth}{
\tikzset{every state/.style={minimum size=2em}} 
\begin{tikzpicture}[->,>=stealth',shorten >=1pt,auto,node distance=3cm,
                    semithick]

  \node[state, fill=gray!40] (A)                    {$a$};
  \node[state]         (C) [above right of=A] {$c$};
  \node[state, fill=gray!40]         (B) [below right of=C] {$b$};

  \path (A) edge  [bend left=10] node {\oz} (C)           
                 edge  [loop left]            node {\zo} (A)
        (B) edge   [loop right] node    {\zo} (B)
            edge         node [swap] {\oz} (A)
           (C) edge [bend left=10] node {\oo} (A)
                 edge [loop above] node {\zz} (C);
        \end{tikzpicture}
}
\bigskip

To see that $c$ has infinite order, observe that
$
a^{-1}(0^\infty)=10^\infty$ and $c^{-1}(0^\infty)=0^\infty$.
Therefore, for all $n\ge 0$, $c^{-n}(10^\infty)\sim0^\infty$.
As $c(10^\infty)=1^\infty$, it follows that $c$ has infinite order.

Note that according to \cite{M} this group is isomorphic to that of automaton $939$
which is non-contracting.  

\subsection{Automaton 2402}
\mbox{}

\begin{minipage}{0.5\textwidth}
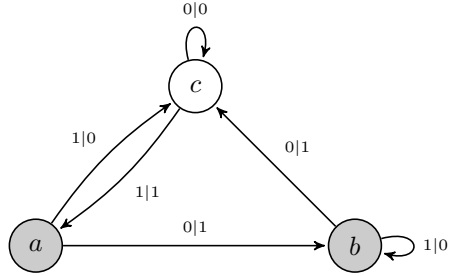
\begin{gather*}
c(0)=0\\c\rest{0}=c
\end{gather*}
\end{minipage}
\parbox{0.4\textwidth}{
\tikzset{every state/.style={minimum size=2em}} 
\begin{tikzpicture}[->,>=stealth',shorten >=1pt,auto,node distance=3cm,
                    semithick]

  \node[state, fill=gray!40] (A)                    {$a$};
  \node[state]         (C) [above right of=A] {$c$};
  \node[state, fill=gray!40]         (B) [below right of=C] {$b$};

  \path (A) edge  [bend left=10] node {\oz} (C)           
                 edge              node {\zo} (B)
        (B) edge   [loop right] node    {\oz} (B)
            edge            node [swap] {\zo} (C)
           (C) edge [bend left=10] node {\oo} (A)
                 edge [loop above] node {\zz} (C);
        \end{tikzpicture}
}
\medskip

Note that
$c^n(10^\infty)\sim 0^\infty$ since $x\rest{00}=c$ for any $x\in\{a,b,c\}$.
However $c^{-2}(10^\infty)=101^\infty$.
Therefore $c$ has infinite order.

\subsection{Automaton 2427}
\mbox{}

\begin{minipage}{0.5\textwidth}
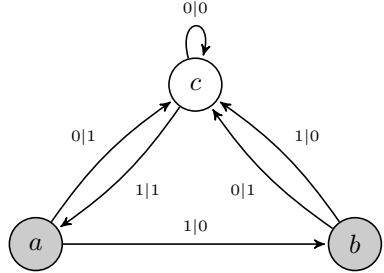
\begin{gather*}
c(0)=0\\c\rest{0}=c
\end{gather*}
\end{minipage}
\parbox{0.4\textwidth}{
\tikzset{every state/.style={minimum size=2em}} 
\begin{tikzpicture}[->,>=stealth',shorten >=1pt,auto,node distance=3cm,
                    semithick]

  \node[state, fill=gray!40] (A)                    {$a$};
  \node[state]         (C) [above right of=A] {$c$};
  \node[state, fill=gray!40]         (B) [below right of=C] {$b$};

  \path (A) edge  [bend left=10] node {\zo} (C)           
                 edge              node {\oz} (B)
        (B) edge   [bend right=10] node  [swap]   {\oz} (C)
            edge      [bend left=10]      node  {\zo} (C)
           (C) edge [bend left=10] node {\oo} (A)
                 edge [loop above] node {\zz} (C);
        \end{tikzpicture}
}
\medskip

To see that $c$ has infinite order note that
$$
a((101)^\infty)=01(101)^\infty \quad b((101)^\infty)=00(101)^\infty \quad\text{and}\quad c((101)^\infty)=11(101)^\infty.
$$
Therefore, for all $n\ge 1$, $c^n((101)^\infty\sim(101)^\infty$.
However, $c^{-2}((101)^\infty)=(100)^\infty$.

\section{Non-hyperbolic self-similarity graphs}\label{sec:hyperbolic}

Nekrashevych introduced the notion of a self-similarity graph of a self-similar action. 
He proved that if a self-similar group is contracting, the corresponding self-similarity graph (endowed with the natural metric) is hyperbolic.
The converse to this result is open.

Here we provide a partial converse to this fact, which applies to self-similar actions that satisfy the criterion of Lemma~\ref{lem:criterion}.
We do not know of a non-contracting self-similar group that does not satisfy the criterion. 

\begin{dfn}[\cite{N} Defn. 3.7.1]
The \emph{self-similarity graph} $\Sigma(G)$ of  a self-similar group $G$ with generating set $S$ acting on $X^{*}$  is the graph with vertex set $X^{*}$ and an  edge $\{u,v\}$ whenever:
\begin{itemize}
\item $u=s(v)$  for some $s\in S$;  these are the {\em horizontal edges};
\item $u=xv$ for some $x\in X$; these are the {\em vertical edges}.
\end{itemize}
\end{dfn}
Observe that horizontal edges connect  strings in $X^*$ of the same length, vertical edges connect 
strings that differ in length by 1.

We use the characterization of hyperbolic geodesic metric spaces involving the divergence of geodesics,
see \cite[p.412]{BH}. 
\begin{prop}
Let $Y$ be a geodesic metric space. A function $e:\N\to\R$ is called a \emph{divergence function }
if for all $y\in Y$, for all $R,r\in\N$
and for all geodesics $\alpha:[0,a]\to Y$ and $\beta:[0,b]\to Y$ with
$\alpha(0)=\beta(0)=y$, $a>R+r$  and $b>R+r$ the following holds:
if $d_Y(\alpha(R),\beta(R))>e(0)$ then any path from $\alpha(R+r)$ to $\beta(R+r)$ that 
stays outside the open ball of radius $R+r$ about $y$ has length at least $e(r)$.
Then $Y$ is hyperbolic if it admits an exponential divergence function.
\end{prop}

\begin{prop}\label{lem:nonhyp}
 Let $G$ be a self-similar group with finite generating set $S$ acting on $X^*$, and 
suppose that there exist $g\in G$ and   $v\in X^*$ such that 
\begin{enumerate}
\item $g\rest{v}=g$
\item $g(v)=v$
\item $g$ has infinite order
\end{enumerate}
Then the self-similarity graph $\Sigma(G)$ is non-hyperbolic. 
\end{prop}

\begin{proof}
The vertex in $\Sigma(G)$ corresponding to the empty string is labelled $\emptyset$.
A vertex in the open ball based at $\emptyset$ of radius $N$ corresponds to a string in $X^*$ of length  less than $N$.
Note that an element  of $X^*$ uniquely defines a vertical geodesic emanating from $\emptyset$ whose length is equal to that of the word.
Considering such geodesics,
we show that $\Sigma(G)$ does not admit an exponential divergence function, and is therefore not hyperbolic.

Suppose for a contradiction that $e:\N\to\R$
is a divergence function for $\Sigma(G)$ and that it is increasing and unbounded.
If the maximum size of an orbit of any $w\in X^*$ under  $g$ was $N$, then $g^{N!}(w)=w$ for all $w\in X^*$. 
Since $g$ has infinite order, it follows that there are arbitrarily large orbits under its action on $X^*$.
Vertices in  
$\Sigma(G)$ have uniformly bounded  degree. 
It follows that there is a bound on the number of vertices in any metric ball of fixed radius, so  we can  choose $n \in \mathbb{N}$ and $w \in X^{*}$ such that $d_\Sigma(w,g^n(w))> e(0)$. 

For all $k\in\N$ the vertices $v^kw$ and $g^n(v^kw)=v^kg^n(w)$ are connected by a horizontal path of length exactly $n||g||_S$, and this path
lies outside the open ball of radius $|v^kw|$ 
 centered at $\emptyset$. Choose $k\in\N$ such that $e(|v^k|)>n||g||_S$.
 Since $e$ is a divergence function, any horizontal path connecting 
 $v^kw$ and $v^kg^n(w)$ must have length at least $e(|v^k|)$. This contradiction establishes the result.
\end{proof}

\printbibliography
\end{document}